\documentclass[leqno,12pt,a4paper,twoside]{article}%
\usepackage{amssymb}
\usepackage{amsfonts}
\usepackage{amsmath}
\usepackage{authblk}
\usepackage{graphicx}
\usepackage{ifthen}
\usepackage{comment}
\usepackage{xcolor}
\usepackage{enumitem}
\usepackage{lipsum,multicol}
\usepackage[T1]{fontenc}
\usepackage[english]{babel}
\usepackage[english]{babel}
\setlist[itemize]{noitemsep, topsep=0pt}
\setlist[enumerate]{noitemsep, topsep=0pt}
\setlist[itemize]{leftmargin=*}
\setlist[enumerate]{leftmargin=*}
\providecommand{\U}[1]{\protect\rule{.1in}{.1in}}

\providecommand{\norm}[1]{\left\lVert#1\right\rVert}
\providecommand{\abs}[1]{\left\lvert#1\right\rvert}
\providecommand{\pr}[1]{\left(#1\right)} 
\providecommand{\pp}[1]{\left[#1\right]} 
\providecommand{\set}[1]{\left\lbrace#1\right\rbrace} 
\providecommand{\scal}[1]{\left\langle#1\right\rangle}

\newcommand{\Lo}[2]{\mathbb{L}^{#1}\pr{{\mathcal{O}_{{#2}}}}}
\newcommand{\Ho}[1]{H_0^1\pr{\mathcal{O}_{#1}}}
\newcommand{\Hmo}[1]{H^{-1}\pr{\mathcal{O}_{#1}}}
\newcommand{\normi}[3]{\norm{#1}_{
		\ifthenelse{\equal{#2}{1}}{H_0^1\pr{\mathcal{O}_{#3}}}{%
			\ifthenelse{\equal{#2}{-1}}{H^{-1}\pr{\mathcal{O}_{#3}}}{}}}}
\newcommand{\X}{X^{t,\xi,\eta,u}}
\newcommand{\x}{x^{t,\xi,\eta,u}}
\newcommand{\xh}{\hat{x}^{t,\xi,\eta,u}}
\newcommand{\y}{y^{t,\xi,\eta,u}}
\newcommand{\yh}{\hat{y}^{t,\xi,\eta,u}}
\newcommand{\Y}{Y^{t,\xi,\eta,u}}

\newcommand{\XX}{\mathbb{X}^{t,\xi,\eta,u}}
\newcommand{\YY}{\mathbb{Y}^{t,\xi,\eta,u}}

\newcommand{\n}[1]{\mathcal{N}\pr{#1}}

\makeatletter
\newcommand{\subjclass}[2][2020]{%
	\let\@oldtitle\@title%
	\gdef\@title{\@oldtitle\footnotetext{\textbf{#1 \emph{Mathematics subject classification.}} #2}}%
}
\newcommand{\keywords}[1]{%
	\let\@@oldtitle\@title%
	\gdef\@title{\@@oldtitle\footnotetext{\textbf{\emph{Key words.}} #1.}}%
}
\makeatother

\newtheorem{theorem}{Theorem}

\newtheorem{definition}[theorem]{Definition}

\newtheorem{proposition}[theorem]{Proposition}
\newtheorem{remark}[theorem]{Remark}

\newenvironment{proof}[1][Proof]{\noindent\textbf{#1.} }{\ \rule{0.5em}{0.5em}}

\title{Asymptotic issues for porous media systems
	with linear multiplicative gradient-type noise via state constrained arguments}
\author[1,2]{Ioana Ciotir}
\author[3,4]{Dan Goreac}
\author[5,6]{Ionu\c t Munteanu\footnote{Corresponding author, email: ionut.munteanu@uaic.ro}}
\affil[1]{\small Normandie University, INSA de Rouen Normandie,
	LMI,  76000 Rouen, France} 
\affil[2]{{ {{Research Center for Pure and App. Math., Graduate School of Information Sciences,
				Tohoku Univ., Japan}}}} 
\affil[3]{School of Mathematics and Statistics, Shandong University, Weihai, Weihai 264209, PR China} 
\affil[4]{LAMA, Univ Gustave Eiffel, UPEM, Univ Paris Est Creteil, CNRS, F-77447 Marne-la-Vallée, France}
\affil[5]{Faculty of Mathematics, Al. I. Cuza University, Bd. Carol I, 11, Iasi 700506, Romania}
\affil[6]{O. Mayer Institute of Mathematics, Romanian Academy, Bd. Carol I, 8, Iasi 700505, Romania}

\date{}
\begin{document}
	\maketitle
	\begin{abstract}
		{The aim of the present paper is to provide necessary and sufficient conditions  to maintain a stochastic coupled system, with porous media components and gradient-type noise in a prescribed set of constraints by using internal controls.  This work is a continuation of the results in \cite{viability},  as we consider the case of divergence type noise perturbation.  On the other hand, it provides a different framework in which the quasi-tangency condition can be obtained with optimal speed. In comparison with the aforementioned result, here we transform the stochastic system into a random deterministic one, via the rescaling approach, then we study the viability of random sets. As an application, conditions for the stabilization of the stochastic porous media equations are obtained.
		}
	\end{abstract} 
	
	\section{Introduction}
	We focus on the following Stratonovich stochastic controlled system
	\begin{equation}\label{e1}\left\{\begin{array}{l}\begin{aligned}dX^{t,\xi,\eta,u}(s)=&\Delta \beta_1(X^{t,\xi,\eta,u}(s))ds+f_1(X^{t,\xi,\eta,u}(s),Y^{t,\xi,\eta,u}(s),u(s))ds\\&
	+B_1X^{t,\xi,\eta,u}\circ dW(s) \text{ in }(t,\infty)\times \mathcal{O}_1,\end{aligned}\\
	\begin{aligned}dY^{t,\xi,\eta,u}(s)=&\Delta \beta_2(Y^{t,\xi,\eta,u}(s))ds+f_2(X^{t,\xi,\eta,u}(s),Y^{t,\xi,\eta,u}(s),u(s))ds\\&
	+B_2Y^{t,\xi,\eta,u}\circ dW(s) \text{ in }(t,\infty)\times \mathcal{O}_2,\end{aligned}\\
	X^{t,\xi,\eta,u}(s)=0 \text{ on }\partial\mathcal{O}_1,\  Y^{t,\xi,\eta,u}(s)=0 \text{ on }\partial\mathcal{O}_2,\\
	X^{t,\xi,\eta,u}(t)=\xi, \  Y^{t,\xi,\eta,u}(t)=\eta,\end{array} \right.\ \end{equation}where, for $i\in\left\{1,2\right\}:$ $\mathcal{O}_i$ is an open, bounded domain in $\mathbb{R}^d,\ d\geq2,$ with smooth boundary $\partial\mathcal{O}_i$; $\beta_i$ and $f_i$ are nonlinear Lipschitz continuous functions (see below); $B_i\phi=b_i\cdot\nabla(-\Delta)^{-1}\phi,$ where $b_i:\overline{\mathcal{O}_i}\rightarrow \mathbb{R}^d$ is the coefficient field and $W$ denotes a one-dimensional violetian motion on a filtered probability space $\left( \Omega, \mathcal{F},\left\{\mathcal{F}_t\right\}_{t\geq0},\mathbb{P}\right)$ obeying the usual assumptions. The initial datum is chosen as $\pr{\xi,\eta}\in L^2(\Omega, \mathcal{F}_t,\mathbb{P};\Hmo{1}\times\Hmo{2}).$ These equations are controlled with a progressively measurable process $u$ taking its values in a compact metric space $U$.  The class of such control policies is denoted by $\mathcal{U}$. The existence and the uniqueness of solutions for the (coupled) system \eqref{e1}  follows in a standard way (see \cite[Chapter 2]{BDPR_2016},  \cite{brezz} or the recent work \cite{HLL2021} for more general coupled systems). For more details about systems of the same type as \eqref{e1} see \cite{viability} and the references therein.
	
	We are interested here on the following problem: under which conditions a given set $K\subset H^{-1}(\mathcal{O}_1)\times H^{-1}(\mathcal{O}_2)$ is viable with respect to equation \eqref{e1}. To be more precise, we introduce the following definition.
	\begin{definition}\label{DefNearViab}
		Given a closed set $K\subset \Hmo{1}\times \Hmo{2}$,  it is said to be \textbf{nearly viable} with respect to \eqref{e1} on the finite interval $\pp{0,T>0}$ if, for every initial time $t\in\left[0,T\right)$ and every initial pair $\pr{\xi,\eta}\in \mathbb{K}_t:=\mathbb{L}^2\pr{\Omega,\mathcal{F}_t,\mathbb{P};K}$, we have \begin{equation}
		\label{EqNearViab'}
		\inf_{u\in \mathcal{U}}\sup_{s\in\pp{t,T}}d\pr{\pr{\X(s),\Y(s)},\mathbb{K}_s}=0.
		\end{equation}
	\end{definition} The {{distance}} is meant in $\mathbb{L}^2\pr{\Omega,\mathcal{F},\mathbb{P};\Hmo{1}\times\Hmo{2}}$.  If an optimal control in \eqref{EqNearViab'} exists, then near viability is equivalent with viability i.e.  $(\X(s),\Y(s))\in K$,  $\mathbb{P}$-a.s. and for all $s\in\pp{t,T}$.  A notion of \textbf{$\mathbb{L}^2$-near viability} can be introduced when $\pr{\xi,\eta}\in \Lo{2}{1}\times\Lo{2}{2}$, $\mathbb{P}$-a.s., by replacing $\mathbb{K}_s$ with $\mathbb{K}_s \cap\mathbb{L}^2\pr{\Omega,\mathcal{F}_s,\mathbb{P};\Lo{2}{1}\times\Lo{2}{2}}$. 
	
	The literature on state-constrained systems is vast. It originates from the work \cite{26}, then it has been extended to tangency concepts to stochastic finite-dimensional systems by \cite{1}. See also the pioneering works \cite{2,5} and the monograph \cite{4}. The results in \cite{16} offer an important semi-group-based method dealing with a wide class of deterministic PDEs.  {Let us also mention the application of such methods to purely deterministic porous media equations in \cite{vrabie}.  Their setting, however is fundamentally different than ours.}\\
	As already mentioned, we shall follow, in the stochastic case, some of the ideas in \cite{viability}.  We emphasize that, in \cite{viability}, the authors study the same problem for a porous media-type system (as \eqref{e1}) but with \emph{nonlinear and Lipschitz} stochastic perturbations.  { The fundamental idea is to replace the solution with a constrained Euler-like scheme and provide good estimates of the distance between the original solution and its approximating scheme. While the best local in time estimate for deterministic systems may be expected to be of order $t$,  \cite[Sections 3.2, 3.3 ]{viability} only exhibit speeds of order $t^{1-\frac{\lambda}{2}}$ for every $\lambda>0$. One way of obtaining the estimate for $\lambda=0$ for the "deterministic" part (i.e. conditional expectations) is to project the system on convenient finite-dimensional spaces, cf.  \cite[Section 3.4]{viability}. We provide here another way of overcoming these technical difficulties and obtain optimal speed when the system can be transformed into a pseudo-deterministic one.}\\
	Partly based on the techniques in \cite{viability}, we provide necessary and sufficient conditions for a set to be viable with respect to \eqref{e1}, in the case of Stratonovich $H^{-1}$-divergence type noise. The choice of a Stratonovich stochastic integral is motivated both by the physical interpretation of the equation {and by the particular affinity to exponential transformations}. One consequence is the fact that the Stratonovich integral is stable with respect to changes in the random term. See \cite{TW} for more details. 
	
	{{The rescaling method which is used in this work was first introduced for stochastic porous media equations in a pioneer paper of Barbu-Röckner (see \cite{first}). 
			Recently,  this technique was intensively used for different stochastic partial differential equations under different sets of assumptions.  See e.g.  \cite{brezz}, \cite{BRflow}, \cite{BRfokker}, \cite{ben1}}}.
	
	{{In the present work}},  via the rescaling argument, we rewrite \eqref{e1} as an equivalent random deterministic system, then deduce the needed energy estimates to conclude with the tangency conditions for viability of the solution. {To avoid redundancy, we will only present the proofs that elude a direct application or minor changes for the arguments in \cite{viability}.} As an application of the theoretical results,  at the end of the paper, we will provide some necessary and sufficient conditions for the exponential asymptotic stabilization of the aforementioned stochastic porous media equation. For other results on long time behaviour of the solution to stochastic porous media equation, see \cite{gess}.
	
	\section{Notations and assumptions}

	For $i\in\set{1,2}$:  $\Lo{p}{i}$, $p\geq 1$ is the standard Banach space of real-valued $p$-power (Lebesgue-) integrable functions on $\mathcal{O}_i$. We set $\|\cdot\|_{\mathbb{L}^p(\mathcal{O}_i) }$ for the classical Lebesgue norm, and set  $\left<\cdot,\cdot\right>_{\Lo{2}{i}}$ for the classical scalar product.  
	$\Ho{i}$ is the space of $\mathbb{L}^2\pr{\mathcal{O}_i}$ functions that vanish on $\partial \mathcal{O}_i$ and such that the distributional derivative of first order belongs to $\mathbb{L}^p\pr{\mathcal{O}_i}$. The norm $\normi{\cdot}{1}{i}$ is given by $\normi{\phi}{1}{i}^2:=\|(-\Delta)^\frac{1}{2}\phi\|^2_{\Lo{2}{i}}.$
	The dual, with pivot space $\mathbb{L}^2(\mathcal{O}_i)$, of the aforementioned space is {{assumed to be}} $H^{-1}\pr{\mathcal{O}_i}$. The associated norm is $\normi{\cdot}{-1}{i}
	$ and the associated product $\scal{\phi,\psi}_{\Hmo{i}}:=\left<(-\Delta)^{-1}\phi,\psi\right>_{\mathbb{L}^2(\mathcal{O}_i)}.$

	Let us introduce $\mathcal{B}_i$, the set of all functions $b=(b^1,b^2,...,b^d),\ b^j:\mathbb{R}^d\rightarrow \mathbb{R},\ j=1,2,...,d,$, such that $b^j\in C^2(\overline{\mathcal{O}_i})$,\ $\text{div} \ b=0$ and $b$ is tangent to the boundary $\partial\mathcal{O}_i$, of the domain $\mathcal{O}_i$. Let any $b\in\mathcal{B}_i$. We associate  the operator $B:\mathbb{L}^2(\mathcal{O}_i)\rightarrow H^{-1}(\mathcal{O}_i)$, defined as
	\begin{equation}\label{e2}B\phi:=b\cdot\nabla (-\Delta)^{-1}\phi, \ \phi\in L^2(\mathcal{O}_i).\end{equation}  
	It is proven in \cite[Section 4.3]{brezz} that {{$B$ }} is densely defined and skew-symmetric in $\Hmo{i}$ {such that $B\in\mathcal{L}(\Hmo{i};\mathbb{L}^2(\mathcal{O}_i))$.} Furthermore, $B$ is the infinitesimal generator of a contraction $C_0-$group in $\Hmo{i}$, denoted by $\left\{e^{sB},\ s\in\mathbb{R}\right\}$. This group leaves $\mathbb{L}^2(\mathcal{O}_i)$ invariant.
	
	In this paper, we assume that $b_i\in \mathcal{B}_i$ for $i\in\set{1,2}$.  For $\phi,\psi \in H^{-1}(\mathcal{O}_i),\ s\in\mathbb{R},$ we have
	\begin{equation}\label{wq5}\left<e^{sB_i}\phi,\psi\right>_{H^{-1}(\mathcal{O}_i)}=\left<\phi,e^{-sB_i}\psi\right>_{H^{-1}(\mathcal{O}_i)}.\end{equation} The contraction property reads
	\begin{equation}\label{wq3}\|e^{sB_i}\phi\|_{\Hmo{i}}\leq \|\phi\|_{\Hmo{i}}.\end{equation}

	Concerning the nonlinearities in \eqref{e1}, we will work under the assumption
	\begin{align}\label{Ass1}\begin{cases}
	(i)\ & \beta_i \text{ is real-valued, Lipschitz continuous of Lipschitz constant }[\beta_i]_i>0;\\
	(ii)\ &\textnormal{There exist } \alpha_i> 0\textnormal{ s.t. }\pr{\beta_i(r)-\beta_i(s)}\pr{r-s}\geq \alpha_i (r-s)^2,\ \forall r,s\in\mathbb{R};\\
	(iii)\ &f_i:\Hmo{1}\times \Hmo{2}\times U\rightarrow \Hmo{i},\\ &\textnormal{There exist } [f_i]_1>0 \textnormal{ such that }\forall x,x'\in\Hmo{1},\ y,y'\in \Hmo{2},\ u\in \mathcal{U},\\ 
	&\norm{f_i(x,y,u)-f_i(x',y',u)}_{\Hmo{i}}\leq \pp{f_i}_{1}\pr{\norm{x-x'}_{\Hmo{1}}+\norm{y-y'}_{\Hmo{2}}};
	\\
	&\underset{u\in \mathcal{U}}{\sup}\norm{f_i(0,0,u)}_{\Hmo{i}}<\infty;\\
	(iv)\ & {\textnormal{ the restriction }f_i:\Lo{2}{1}\times \Lo{2}{2}\times U\rightarrow \Lo{2}{i}\textnormal{ enjoys the same properties}.}
	\end{cases}
	\end{align}
	
	\section{The equivalent formulation and the main result}

	Let us introduce  $\Gamma_i, \ i=1,2,$  the stochastic exponential $\Gamma_i(t,s)=e^{[W(s)-W(t)]B_i},\ i=1,2.$ 	Following the ideas in \cite{brezz}, by the rescaling $X^{t,\xi,\eta,u}(s)=\Gamma_1(t,s)x^{t,\xi,\eta,u}(s)$ and $Y^{t,\xi,\eta,u}(s)=\Gamma_2(t,s)y^{t,\xi,\eta,u}(s),$ we equivalently rewrite  \eqref{e1} as a random-deterministic system
	\begin{equation}\label{e3}\left\{\begin{array}{l}\begin{aligned}\partial_s x^{t,\xi,\eta,u}(s)=&\Gamma_1(s,t)\Delta \beta_1(\Gamma_1(t,s)x^{t,\xi,\eta,u}(s))\\&+\Gamma_1(s,t)f_1(\Gamma_1(t,s)x^{t,\xi,\eta,u}(s),\Gamma_2(t,s)y^{t,\xi,\eta,u}(s),u(s)) \text{ in }(t,\infty)\times \mathcal{O}_1,\end{aligned}\\
	\begin{aligned}\partial_sy^{t,\xi,\eta,u}(s)=&\Gamma_2(s,t)\Delta \beta_2(\Gamma_2(t,s)y^{t,\xi,\eta,u}(s))\\&
	+\Gamma_2(s,t)f_2(\Gamma_1(t,s)x^{t,\xi,\eta,u}(s),\Gamma_2(t,s)y^{t,\xi,\eta,u}(s),u(s))\text{ in }(t,\infty)\times \mathcal{O}_2,\end{aligned}\\
	x^{t,\xi,\eta,u}(s)=0 \text{ on }\partial\mathcal{O}_1,\  y^{t,\xi,\eta,u}(s)=0 \text{ on }\partial\mathcal{O}_2,\\
	x^{t,\xi,\eta,u}(t)=\xi, \  y^{t,\xi,\eta,u}(t)=\eta.\end{array} \right.\end{equation}

	One can easily adapt Definition  \ref{DefNearViab} in order to precise what a viable set  to system \eqref{e3} means. Then, 
	owing to \eqref{wq5}-\eqref{wq3},	it is enough to study the viability of equation \eqref{e3} only, because
	$$
	\inf_{u\in \mathcal{U}}\sup_{s\in\pp{t,T}}d\pr{\pr{\X(s),\Y(s)},\mathbb{K}_s}\leq \inf_{u\in \mathcal{U}}\sup_{s\in\pp{t,T}}d\pr{\pr{x^{t,\xi,\eta,u}(s),y^{t,\xi,\eta,u}(s)},\hat{\mathbb{K}}_{t,s}}
	$$
	where $\hat{\mathbb{K}}_{t,s}:=\left(\begin{array}{cc}\Gamma_1(s,t)&0\\ 0 &\Gamma_2(s,t)\end{array}\right)\mathbb{K}_s$. Here,  $d$ is the same distance as in Definition \ref{DefNearViab}. {Since we are interested in norm considerations to evaluate an asymptotic behaviour,  and in order to avoid moving sets, we assume that $\hat{\mathbb{K}}_{t,s}=\mathbb{K}_s$ is no-longer time dependent (except for unavoidable adaptedness issues). { But, since we deal with a random-deterministic equation, and due to the particular structure of $\Gamma_i$ depending only on the increment of the Brownian, adaptedness issues can be abandoned altogether. We will simply write $\mathbb{K}$ instead of $\mathbb{K}_T$.}
		
		Before moving on, let us recall that  {$B_i\in \mathcal{L}(H^{-1},\mathbb{L}^2)$},  {such that \[\norm{e^{sB_i}-\mathbb{I}}_{\mathcal{L}\pr{\Hmo{i};\Lo{2}{i}}}\leq e^{\abs{s}\norm{B_i}_{\mathcal{L}\pr{\Hmo{i};\Lo{2}{i}}}}-1. \] 
			{
				One recalls that, for the folded normal distribution, one has \[\int_{\mathbb{R}}\pr{e^{\delta\abs{r}}-1}\mathbb{P}_{\mathcal{N}(0,1)}(dr)=2e^{\frac{\delta^2}{2}}\int_{-\infty}^\delta\frac{1}{\sqrt{2\pi}}e^{-\frac{l^2}{2}}dl-1:=\omega(\delta). \]
				As a consequence, whenever $\zeta\in\mathbb{L}^2\pr{\Omega,\mathcal{F}_t,\mathbb{P};\Hmo{i}}$, and $s\geq t$, by using the independence of $W(s)-W(t)$ of $\mathcal{F}_t$ and the stationarity of the increments, one has
				\begin{equation}\label{bun''}\begin{split}
				&\mathbb{E}\pp{\norm{\Gamma_i(t,s)\zeta-\zeta}_{\Lo{2}{i}}^2}\leq \int_{\mathbb{R}}\pr{e^{\abs{r}\norm{B_i}_{\mathcal{L}\pr{\Hmo{i};\Lo{2}{i}}}}-1}^2\mathbb{P}_{W(s-t)}(dr)\mathbb{E}\pp{\norm{\zeta}^2_{\Hmo{i}}}\\
				\leq &\int_{\mathbb{R}}\pr{e^{2\abs{r}\norm{B_i}_{\mathcal{L}\pr{\Hmo{i};\Lo{2}{i}}}}-1-2\pr{e^{\abs{r}\norm{B_i}_{\mathcal{L}\pr{\Hmo{i};\Lo{2}{i}}}}-1}}\mathbb{P}_{W(s-t)}(dr)\mathbb{E}\pp{\norm{\zeta}^2_{\Hmo{i}}}\\
				=&\mathbb{E}\pp{\norm{\zeta}^2_{\Hmo{i}}}\pp{\omega(2\delta)-2\omega(\delta)}\mid_{\delta=\sqrt{s-t}\norm{B_i}_{\mathcal{L}\pr{\Hmo{1};\Lo{2}{1}}}}.
				\end{split}\end{equation}
				The reader is invited to note that $\frac{\omega(2\delta)-2\omega(\delta)}{\delta^2}$ is bounded (arround $0+$). Indeed, \begin{align*}\omega(\delta)=O(\delta^2)+2\int_{-\infty}^\delta \frac{1}{\sqrt{2\pi}}e^{-\frac{l^2}{2}}dl-2\int_{-\infty}^0 \frac{1}{\sqrt{2\pi}}e^{-\frac{l^2}{2}}dl\\
				=O\pr{\delta^2}+\frac{2}{\sqrt{2\pi}}\int_0^\delta \pr{1-\frac{l^2}{2}}dl=O\pr{\delta^2}+\frac{2\delta}{\sqrt{2\pi}}.\end{align*}As a consequence,  $\omega(2\delta)-2\omega(\delta)=O(\delta^2)$ (with the obvious use of Landau notation), and we get the existence of a generic universal constant $C>0$ (independent of time and the initial datum $\zeta$) for which
				\begin{equation}\label{bun'}
				\mathbb{E}\pp{\norm{\Gamma_i(t,s)\zeta-\zeta}_{\Lo{2}{i}}^2}\leq C\mathbb{E}\pp{\norm{\zeta}^2_{\Hmo{i}}}(s-t),\ \forall s\geq t \textnormal{ small enough}.
				\end{equation}}
			Let us consider the fundamental solution (acting as an Euler-type scheme) associated to \eqref{e3}
			\begin{equation}\label{wq8}\left\{\begin{array}{l}\begin{aligned}\partial_s \hat{x}^{t,\xi,\eta,u}(s)=&\Gamma_1(s,t)\Delta \beta_1(\Gamma_1(t,s)\hat{x}^{t,\xi,\eta,u}(s))+f_1(\xi,\eta,u(s)),\end{aligned}\\
			\begin{aligned}\partial_s\hat{y}^{t,\xi,\eta,u}(s)=&\Gamma_2(s,t)\Delta \beta_2(\Gamma_2(t,s)\hat{y}^{t,\xi,\eta,u}(s))
			+f_2(\xi,\eta,u(s)),\end{aligned}\\
			\hat{x}^{t,\xi,\eta,u}(s)=0 \text{ on }\partial\mathcal{O}_1,\ \hat{ y}^{t,\xi,\eta,u}(s)=0 \text{ on }\partial\mathcal{O}_2,\ s\geq t,\\
			\hat{x}^{t,\xi,\eta,u}(t)=\xi\in\mathbb{L}^2\pr{\Omega,\mathcal{F}_t,\mathbb{P};\Hmo{1}}, \  \hat{y}^{t,\xi,\eta,u}(t)=\eta\mathbb{L}^2\pr{\Omega,\mathcal{F}_t,\mathbb{P};\Hmo{2}}.\end{array} \right.\end{equation} 
			
			In order to state the main result of this paper, let us introduce below the definition of quasi-tangent sets
			\begin{definition}\label{DefLambdaQT}
				Let us fix the finite interval $\pp{0,T}$ and $t\in\left[0,T\right)$.  A closed set $K\subset \Hmo{1}\times \Hmo{2}$ satisfies the \textbf{quasi-tangency condition} with respect to the control system \eqref{e3} at $\pr{\xi,\eta}\in \mathbb{K}_t$ if \begin{align*}
				\underset{\varepsilon\rightarrow0+}{\lim\inf}\inf\big\lbrace&\frac{1}{\varepsilon^2}\mathbb{E}\pp{\norm{\hat{x}^{t,\xi,\eta,u}(t+\varepsilon)-\theta_1}_{\Hmo{1}}^2+\norm{\hat{y}^{t,\xi,\eta,u}(t+\varepsilon)-\theta_2}_{\Hmo{2}}^2
				}:\\ &\pr{\theta_1,\theta_2}\in \mathbb{K}_{t+\varepsilon}\big\rbrace=0,\ \mathbb{P}-\text{a.s.}.
				\end{align*} 
				If this condition holds for every $t\in\left[0,T\right)$ and every $\pr{\xi,\eta}\in \Hmo{1}\times \Hmo{2}$, then we will simply say that $K$ satisfies the quasi-tangency condition. {This corresponds to \cite[Definition 4]{viability} with $\lambda=0$.}
				
			\end{definition}
			The main theoretical result is the following.
			\begin{theorem}\label{Tq1}
				Under the above assumptions, the closed set $K$ is near viable with respect to \eqref{e3} if and only if $K$ is quasi-tangent for every $t\in[0,T)$ and every initial data $\xi,\eta$ belonging to the space $\mathbb{K}_t$.
			\end{theorem}
			{
				\begin{proof}The sufficiency follows immediately by invoking the results in \cite[Theorem 9 and Theorem 18]{viability}, forwardly adapted to the random-deterministic case.  {For our readers' sake, we provide the adaptation of the proof to this framework in the Appendix \ref{A1}.} For the necessity, which constitutes the whole contribution of this paper, we invite the reader to take a look at Proposition \ref{prop1} and Proposition \ref{prop2} below.
				\end{proof}

				First, we will need some estimates on the right-continuity of schemes which are sharper than the ones obtained in \cite[Proposition 5]{viability}.
				\begin{proposition}\label{prop1}Let the time horizon $T>0$ be fixed.  Then, there exists a constant $C>0$ (generic, dependent of $T$), but not on the initial data such that, for every $\xi\in\mathbb{L}^2\pr{\Omega,\mathcal{F}_t,\mathbb{P};\Hmo{1}}$ and every $\eta\in\mathbb{L}^2\pr{\Omega,\mathcal{F}_t,\mathbb{P};\Hmo{2}}$, and every $t\leq s\leq T$, we have
					\begin{equation}\label{ie11}\begin{cases}
					&\mathbb{E}\pp{\normi{\xh(s)-\xi}{-1}{1}^2}\leq C\pr{1+\mathbb{E}\pp{\norm{\beta_1\pr{\xi}}_{\Ho{1}}^2+\norm{\xi}_{\Hmo{1}}^2+\norm{\eta}_{\Hmo{2}}^2}}^{\frac{3}{2}}(s-t)^{\frac{3}{2}},\\
					&\mathbb{E}\pp{\normi{\yh(s)-\eta}{-1}{1}^2}\leq C\pr{1+\mathbb{E}\pp{\norm{\beta_2\pr{\eta}}_{\Ho{2}}^2+\norm{\xi}_{\Hmo{1}}^2+\norm{\eta}_{\Hmo{2}}^2}}^{\frac{3}{2}}(s-t)^{\frac{3}{2}}.
					\end{cases}
					\end{equation}
					
				\end{proposition}
				
				\begin{proof}
					By \eqref{wq8} and \eqref{wq5},  we have
					\begin{align*}
					\frac{d}{ds} \normi{\xh-\xi}{-1}{1}^2=&2\scal{\Gamma_1(s,t)\Delta\beta_1\pr{\Gamma_1(t,s)\xh},\xh-\xi}_{\Hmo{1}}\\&+2\scal{\Gamma_1(s,t)f_1\pr{\Gamma_1(t,s)\xh,\Gamma_2(t,s)\yh,u(s)},\xh-\xi}_{\Hmo{1}}.
					\end{align*}
					For the first term, we use the monotonicity of $\beta_1$ and the duality products to get
					\begin{align*}
					& \scal{\Gamma_1(s,t)\Delta\beta_1\pr{\Gamma_1(t,s)\xh(s)},\xh(s)-\xi}_{\Hmo{1}}\\&=\scal{\Delta\pr{\beta_1\pr{\Gamma_1(t,s)\xh(s)}-\beta_1\pr{\Gamma_1(t,s)\xi}},\Gamma_1(t,s)\pr{\xh(s)-\xi}}_{\Hmo{1}}\\&-\scal{\beta_1\pr{\Gamma_1(t,s)\xi}-\beta_1(\xi),\Gamma_1(t,s)\pr{\xh(s)-\xi}}_{\Lo{2}{1}}\\&-\scal{\beta_1\pr{\xi},\Gamma_1(t,s)\pr{\xh(s)-\xi}}_{\Ho{1},\Hmo{1}}\\
					&\leq -\alpha\norm{\Gamma_1(t,s)\xh(s)-\Gamma_1(t,s)\xi}_{\Lo{2}{1}}^2\\
					&+\pp{\beta_1}_1\norm{\Gamma_1(t,s)-\mathbb{I}}_{\mathcal{L}\pr{\Hmo{1};\Lo{2}{1}}}\norm{\xi}_{\Hmo{1}}\norm{\Gamma_1(t,s)\xh(s)-\Gamma_1(t,s)\xi}_{\Lo{2}{1}}\\
					&+\norm{\beta_1\pr{\xi}}_{\Ho{1}}\norm{\Gamma_1(t,s)\pr{\xh(s)-\xi}}_{\Hmo{1}}.
					\end{align*}
					Using classical inequalities of type $ab\leq \frac{1}{4\delta}a^2+\delta b^2$ and the contraction property \eqref{wq3}, one gets 
					\begin{align*}
					& \scal{\Gamma_1(s,t)\Delta\beta_1\pr{\Gamma_1(t,s)\xh(s)},\xh(s)-\xi}_{\Hmo{1}}\\&\leq C\pp{\beta_1}_1^2\norm{\Gamma_1(t,s)-\mathbb{I}}_{\mathcal{L}\pr{\Hmo{1};\Lo{2}{1}}}^2\norm{\xi}_{\Hmo{1}}^2+\norm{\beta_1\pr{\xi}}_{\Ho{1}}\norm{\xh(s)-\xi}_{\Hmo{1}}.
					\end{align*}
					The term in $f_1$ is treated similarly due to Lipschitz-continuity and growth bounds. Putting these together and taking expectancy,  it follows (again due to the independence of $\Gamma_1(t,s)$ of $\mathcal{F}_t$) that
					\begin{equation}
					\label{estim0}\begin{split}
					&\mathbb{E}\pp{\normi{\xh(s)-\xi}{-1}{1}^2}\\\leq &C\pp{\beta_1}_1^2\int_t^s\mathbb{E}\pp{\norm{\Gamma_1(t,l)-\mathbb{I}}_{\mathcal{L}\pr{\Hmo{1};\Lo{2}{1}}}^2}dl\mathbb{E}\pp{\norm{\xi}_{\Hmo{1}}^2}\\
					&+C\int_t^s\mathbb{E}\pp{\norm{\beta_1\pr{\xi}}_{\Ho{1}}\norm{\xh(l)-\xi}_{\Hmo{1}}}dl\\
					&+C\int_t^s\mathbb{E}\pp{\pr{1+\norm{\xi}_{\Hmo{1}}+\norm{\eta}_{\Hmo{2}}}\norm{\xh(l)-\xi}_{\Hmo{1}}}dl.\\
					\leq &C\pp{\beta_1}_1^2\int_t^s\mathbb{E}\pp{\norm{\Gamma_1(t,l)-\mathbb{I}}_{\mathcal{L}\pr{\Hmo{1};\Lo{2}{1}}}^2}dl\mathbb{E}\pp{\norm{\xi}_{\Hmo{1}}^2}\\
					&+C\pr{1+\mathbb{E}\pp{\norm{\beta_1\pr{\xi}}_{\Ho{1}}^2+\norm{\xi}_{\Hmo{1}}^2+\norm{\eta}_{\Hmo{2}}^2}}\int_t^s\pr{\mathbb{E}\pp{\norm{\xh(l)-\xi}_{\Hmo{1}}^2}}^{\frac{1}{2}}dl,
					\end{split}
					\end{equation}
					where we have used Cauchy-type inequalities for the expectation in the last two terms.  In a first step, by using $\sqrt{a}\leq \frac{1}{2}(1+a)$ for the last term, it holds that 
					\[\mathbb{E}\pp{\normi{\xh(s)-\xi}{-1}{1}^2}\leq C\pr{1+\mathbb{E}\pp{\norm{\beta_1\pr{\xi}}_{\Ho{1}}^2+\norm{\xi}_{\Hmo{1}}^2+\norm{\eta}_{\Hmo{2}}^2}}(s-t).\]By plugging this into \eqref{estim0} and recalling that \eqref{bun'} holds true, we get the desired result. The second assertion is quite similar.
			\end{proof}}

			Next, we shall compare  \eqref{e3}, 
			with its Euler's approximation \eqref{wq8}. To this end, we have the following result
			
			\begin{proposition}\label{prop2} There exists a constant $C>0$ (depending only on the time horizon $T>0$ and $\omega\in \Omega$) such that, for every $\xi\in\mathbb{L}^2\pr{\Omega,\mathcal{F}_t,\mathbb{P};\Lo{2}{1}}$, every $\eta\in\mathbb{L}^2\pr{\Omega,\mathcal{F}_t,\mathbb{P};\Lo{2}{2}}$,  every admissible control $u\in \mathcal{U}$, and every $t\leq s\leq T$,
				\begin{equation}\begin{aligned}&\mathbb{E}\pp{\|x^{t,\xi,\eta,u}(s)-\hat{x}^{t,\xi,\eta,u}(s)\|^2_{\Hmo{1}}+\|y^{t,\xi,\eta,u}(s)-\hat{y}^{t,\xi,\eta,u}(s)\|^2_{\Hmo{2}}}\\&\leq C\pr{1+\mathbb{E}\pp{\|\beta_1(\xi)\|^2_{\Ho{1}}+\|\beta_2(\eta)\|^2_{\Ho{2}}+\norm{\xi}_{\Hmo{1}}^2+\norm{\eta}_{\Hmo{2}}^2}}^{3}(s-t)^{\frac{9}{4}}.\end{aligned}
				\end{equation}
				
			\end{proposition}
			\begin{proof}
				The proof holds immediately from Proposition \ref{prop1} by arguing similarly to \cite[Proposition 8]{viability}. { Let us just give a sketch of proof.  As usual, one begins with writing down the differential formula for $\normi{\x(s)-\xh(s)}{-1}{1}^2$ for $t\leq s\leq T$ and distinguishes two terms.
					1. For the term in $\Delta\beta_1$, one has
					\begin{align*}
					&\scal{\Gamma_1(s,t)\pr{\Delta\beta_1\pr{\Gamma_1(t,s)\x(s)}-\Delta\beta_1\pr{\Gamma_1(t,s)\xh(s)}},\x(s)-\xh(s)}_{\Hmo{1}}\\&=-\scal{\beta_1\pr{\Gamma_1(t,s)\x(s)}-\beta_1\pr{\Gamma_1(t,s)\xh(s)},\Gamma_1(t,s)\x(s)-\Gamma_1(t,s)\xh(s)}_{\Lo{2}{1}}
					\\&\leq -\alpha_1\norm{\Gamma_1(t,s)\x(s)-\Gamma_1(t,s)\xh(s)}_{\Lo{2}{1}}^2.
					\end{align*}
					Similarly, the terms involving $f_1$ are dealt with as follows (recalling that $e^{rB_1}$ is a contraction in $\Hmo{1}$), \begin{align*}
					&\big\langle \Gamma_1(s,t)f_1\pr{\Gamma_1(t,s)\x(s),\Gamma_2(t,s)\y(s),u(s)}-f_1\pr{\xi,\eta,u(s)},\\ & \ \ \ \ \ \x(s)-\xh(s)\big\rangle_{\Hmo{1}}\\
					&\leq \norm{\x(s)-\xh(s)}_{\Hmo{1}}\norm{\pr{\Gamma_1(s,t)-1}f_1\pr{\xi,\eta,y(s)}}_{\Hmo{1}}
					\\&+\norm{\x(s)-\xh(s)}_{\Hmo{1}}\pp{f_1}_1\pr{\norm{\pr{\Gamma_1(t,s)-1}\xi}_{\Hmo{1}}+\norm{\x(s)-\xi}_{\Hmo{1}}}\\&+\norm{\x(s)-\xh(s)}_{\Hmo{1}}\pp{f_1}_1\pr{\norm{\pr{\Gamma_2(t,s)-1}\eta}_{\Hmo{2}}+\norm{\y(s)-\eta}_{\Hmo{2}}}.
					\end{align*}
					One intercalates $\xh$ and $\yh$ to have
					\begin{align*}
					&\big\langle \Gamma_1(s,t)f_1\pr{\Gamma_1(t,s)\x(s),\Gamma_2(t,s)\y(s),u(s)}-f_1\pr{\xi,\eta,u(s)},\\ & \ \ \ \ \ \x(s)-\xh(s)\big\rangle_{\Hmo{1}}\\
					&\leq \norm{\x(s)-\xh(s)}_{\Hmo{1}}\norm{\pr{\Gamma_1(s,t)-1}f_1\pr{\xi,\eta,y(s)}}_{\Hmo{1}}
					\\&+\norm{\x(s)-\xh(s)}_{\Hmo{1}}\pp{f_1}_1\pr{\norm{\pr{\Gamma_1(t,s)-1}\xi}_{\Hmo{1}}+\norm{\xh(s)-\xi}_{\Hmo{1}}}\\&+\norm{\x(s)-\xh(s)}_{\Hmo{1}}\pp{f_1}_1\pr{\norm{\pr{\Gamma_2(t,s)-1}\eta}_{\Hmo{2}}+\norm{\yh(s)-\eta}_{\Hmo{2}}}\\
					&+\pp{f_1}_1\norm{\x(s)-\xh(s)}_{\Hmo{1}}^2\\
					&+\pp{f_1}_1\norm{\x(s)-\xh(s)}_{\Hmo{1}}\norm{\y(s)-\yh(s)}_{\Hmo{2}}.
					\end{align*}
					Employing Cauchy-type inequalities and \eqref{ie11} and \eqref{bun'}, one has
					\begin{align*}
					&\int_t^s\mathbb{E}\pp{\norm{\x(r)-\xh(s)}_{\Hmo{1}}\norm{\pr{\Gamma_1(t,r)-1}\xi}_{\Hmo{1}}}dr\\&\leq \int_t^s \pr{\pr{\mathbb{E}\pp{\norm{\x(r)-\xh(r)}_{\Hmo{1}}^2}}^{\frac{1}{2}}\pr{\mathbb{E}\pp{\norm{\pr{\Gamma_1(t,r)-1}\xi}_{\Hmo{1}}^2}}^{\frac{1}{2}}}dr\\
					&\leq C\pr{1+\mathbb{E}\pp{\norm{\xi}_{\Hmo{1}}^2}}(s-t)^{\frac{1}{2}}\int_t^s \pr{\mathbb{E}\pp{\norm{\x(r)-\xh(r)}_{\Hmo{1}}^2}}^{\frac{1}{2}}dr.
					\end{align*}
					A similar reasoning holds true for $\xh-\xi$ replacing $(\Gamma_1(t,s)-1)\xi$, owing to Proposition \ref{prop1} with the corresponding constant.  The same type of consideration holds true for $\norm{\y-\yh}_{\Hmo{2}}^2$. Adding all these terms, one has, with a constant \[C:=C\pr{1+\mathbb{E}\pp{\|\beta_1(\xi)\|^2_{\Ho{1}}+\|\beta_2(\eta)\|^2_{\Ho{2}}+\norm{\xi}_{\Hmo{1}}^2+\norm{\eta}_{\Hmo{2}}^2}}^{\frac{3}{2}},\]the following
					\begin{align*}
					\Phi(s)&:=\mathbb{E}\pp{\norm{\x(s)-\xh(s)}_{\Hmo{1}}^2+\norm{\y(s)-\yh(s)}_{\Hmo{2}}^2}\\
					&\leq C(s-t)^{\frac{1}{2}}\int_t^s\pr{\mathbb{E}\pp{\norm{\x(r)-\xh(r)}_{\Hmo{1}}^2}}^{\frac{1}{2}}dr\\
					&+C(s-t)^{\frac{1}{2}}\pr{\mathbb{E}\pp{\norm{\y(r)-\yh(r)}_{\Hmo{2}}^2}}^{\frac{1}{2}}dr+C\int_t^s\Phi(r)dr.
					\end{align*}
					Applying Gronwall's inequality combined with the inequality $\sqrt{a}\leq \frac{1}{2}(1+a)$, one gets $\Phi(s)\leq C(s-t)^{\frac{3}{2}}$. Then, by substituting this into the first term on the righ side, and invoking,once again Gronwall's inequality, one gets the desired estimates.}
			\end{proof}

			{As a consequence,  we have that, if $\pr{\x,\y}$ remains in the set of constraints (with some $u\in\mathcal{U}$),  then, Proposition \ref{prop2} applied for $s=t+\varepsilon$ yields that $\pr{\theta_1^\varepsilon,\theta_2^\varepsilon}:=\pr{\x(t+\varepsilon),\y(t+\varepsilon)}$ constitute the good choice in Definition \ref{DefLambdaQT}. This proves the necessity of the quasi-tangency condition in Theorem \ref{Tq1}.}
			\section{Application} Let us notice that, owing to the isometry \eqref{wq3}, we have that the solution $X$ to \eqref{e1} and $x$ to \eqref{e3} satisfy $\|X(s)\|_{H^{-1}(\mathcal{O}_1)}=\|x(s)\|_{H^{-1}(\mathcal{O}_1)}$, due to the relation $X(s)=\Gamma_1(t,s)x(s).$ Therefore, it suffices to study the stability of the random deterministic porous media equation. We shall apply the ideas from \cite[Section 5]{viability}.  To this end, we consider in \eqref{e3} the case in which $\beta_2\equiv 0$ and $f_2(x,y,u)=-cy,\ c>0,$ and the set
			$$K=\left\{(\xi,\eta)\in H^{-1}(\mathcal{O}_1)\times\mathbb{R}:\ \sum_{k=1}^\infty\left<\xi,\sqrt{\lambda_k}e_k\right>^2_{H^{-1}(\mathcal{O}_1)}\leq \eta\right\},$$where $\left\{\lambda_k\right\}_{k\in\mathbb{N}^*}$ is the set of eigenvalues of the Dirichlet-Laplace operator in $\mathbb{L}^2(\mathcal{O}_1)$ with the corresponding eigenfunctions $\left\{e_k\right\}_{k\in\mathbb{N}^*}$. The viability of such sets implies that  $\|x^{0,\xi,\eta,u}(t)\|^2_{H^{-1}(\mathcal{O}_1)}\leq \eta e^{-ct}$ $\mathbb{P}-$a.s. provided that  $\|\xi\|^2_{H^{-1}(\mathcal{O}_1)}\leq\eta.$ 
			
			Following the arguments in \cite[Eq. (31)]{viability}, we deduce a \textbf{necessary} condition for invariance of $K$:
			$$\inf_{u\in U}\left\{2\sum_{k=1}^\infty\left<x,e_k\right>_{H^{-1}(\mathcal{O}_1)}\left<-\lambda_k\beta(x)+f_1(x,\norm{x}^2_{H^{-1}(\mathcal{O}_1)},u),e_k\right>_{H^{-1}(\mathcal{O}_1)}+2c\|x\|^2_{H^{-1}(\mathcal{O}_1)}\right\}\leq 0,$$for all $x\in \Hmo{1}$. 
			{We emphasize that this cannot be inferred from \cite{viability}, but it is due to the new and stronger necessary condition.  The proof, however, is similar to what is done in \cite[Section 5]{viability}.\\
				Using exactly the same reasoning as the one in \cite[Section 5]{viability}, it can be shown that a sufficient condition for an uncontrolled system to be $c$-exponentially stable is \begin{align*}&\psi_j(x):=\\&2\sum_{k=1}^j\left<x,e_k\right>_{H^{-1}(\mathcal{O}_1)}\left<-\lambda_k\beta(x)+f_1(x,\sum_{l=1}^j\scal{x,e_l}^2_{H^{-1}(\mathcal{O}_1)}),e_k\right>_{H^{-1}(\mathcal{O}_1)}+2c\sum_{l=1}^j\scal{x,e_l}^2_{H^{-1}(\mathcal{O}_1)}\leq 0,\end{align*}for all $j$ large enough.\\
				For controlled coefficients $f_1$, the aforementioned condition concerns a $\psi_j(x,u)$ and needs to exhibit $u$ independent of $j$ i.e. 
				$$U(x):=\set{u\in U:\ \psi_j(x,u)\leq 0,\ \forall j\geq 1}\neq\emptyset,$$
				and the set-valued function $x\mapsto U(x)$ needs to admit a measurable selection. 
			}
			{
				\section{Appendix}
				\subsection{The $\varepsilon$-approximate approach to sufficiency in Theorem \ref{Tq1}}\label{A1}
				\begin{definition}\label{DefEpsSol}
					For $0\leq t\leq T$,  initial data $\pr{\xi,\eta}\in\mathbb{K}_t\cap\mathbb{L}^2\pr{\Omega,\mathcal{F}_t,\mathbb{P};\Lo{2}{1}\times\Lo{2}{2}}$ and $\varepsilon>0$, we will call \emph{constrained $\varepsilon$-approximate solution} to \eqref{e3} a 7-uple $\pr{\bar T,\tau,u,\phi_1,\phi_2,\pr{\mathcal{X},\mathcal{Y}}}$ satisfying
					\begin{enumerate}
						\item $t\leq\bar{T}\leq T$;
						\item the measurable $\tau:\pp{t,\bar{T}}\rightarrow\pp{t,\bar{T}}$ is non-decreasing, non-anticipating and at most $\varepsilon$-delayed i.e. 
						$s-\varepsilon\leq\tau(s)\leq s,\ \forall s\in\pp{t,\bar{T}}$;
						\item the control $u\in\mathcal{U}$;
						\item the corrections $\pr{\phi_1,\phi_2}:\pp{t,\bar{T}}\rightarrow \Hmo{1}\times \Hmo{2}$
						are $\mathcal{F}_{\bar T}$-measurable, take their values in $\Lo{2}{i}$ $\mathbb{P}$-a.s.  and
						$\mathbb{E}\pp{\int_t^{\bar T}\normi{\phi_i(l)}{-1}{i}^2dl}\leq \varepsilon\pr{\bar T-t}$;
						\item the processes $\mathcal{X}, \mathcal{Y}$ are $\mathcal{F}_{\bar T}$-measurable, $\Lo{2}{1}\times\Lo{2}{2}$-valued,  and satisfy (in the classical $\Hmo{i}$-sense) \begin{align}\label{EqXYCal}
						\begin{cases}
						d\mathcal{X}(s)&=\pr{\Gamma_1(s,t)\Delta\beta_1\pr{\Gamma_1(t,s)\mathcal{X}(s)}+f_1\pr{\mathcal{X}(\tau(s)),\mathcal{Y}(\tau(s)),u(s)}+\phi_1(s)}ds\\[4pt]
						d\mathcal{Y}(s)&=\pr{\Gamma_2(s,t)\Delta\beta_2\pr{\Gamma_2(t,s)\mathcal{Y}(s)}+f_2\pr{\mathcal{X}(\tau(s)),\mathcal{Y}(\tau(s)),u(s)}+\phi_2(s)}ds,\ s\geq t,\\[4pt]
						\mathcal{X}(t)&=\xi,\ \mathcal{Y}(t)=\eta.
						\end{cases}
						\end{align}
						\item For every $s\in\pp{t,\bar{T}}$, the constraint $\pr{\mathcal{X}(\tau(s)),\mathcal{Y}(\tau(s))}\in K$, $\mathbb{P}$-a.s., $\pr{\mathcal{X}(\bar T),\mathcal{Y}(\bar T)}\in \mathbb{K}_{\bar T}$, $\mathbb{P}$-a.s. and 
						\begin{equation*}
						\mathbb{E}\pp{\normi{\mathcal{X}(\tau(s))-\mathcal{X}(s)}{-1}{1}^2+\normi{\mathcal{Y}(\tau(s))-\mathcal{Y}(s)}{-1}{2}^2}\leq \varepsilon,\ \forall s\in\pp{t,\bar T}.
						\end{equation*}
					\end{enumerate}
				\end{definition}
				\begin{remark}
					The careful reader will have noticed that although we deal with deterministic equations, because of the randomness contained in $\Gamma_i$, we still impose norms related to $\mathbb{L}^2\pr{\Omega,\mathcal{F},\mathbb{P}}$, but, unlike the true stochastic case in \cite{viability}, anticipating solutions $\mathcal{X},\mathcal{Y}$ are not prohibited. The only measurability imposed is with respect to the time horizon $\bar{T}$.
				\end{remark}
				\begin{proposition}\label{PropPropCal}
					There exists a constant $C>0$ such that, for $t\leq s\leq \bar T$ and \[\pr{\xi,\eta}\in\mathbb{L}^2\pr{\Omega,\mathcal{F}_t,\mathbb{P};K\cap\pr{\Lo{2}{1}\times\Lo{2}{2}}}, \]one has the following properties of an $\varepsilon$-approximate solution as described in Definition \ref{DefEpsSol}.
					\begin{enumerate}
						\item $\pr{\mathcal{X},\mathcal{Y}}$ has a modification belonging to $\Lo{2}{1}\times\Lo{2}{2}$.
						\item This modification (still denoted by $\pr{\mathcal{X},\mathcal{Y}}$) is continuous as a time function with values in 
						$\mathbb{L}^2\pr{\Omega,\mathcal{F},\mathbb{P};\Hmo{1}\times\Hmo{2}}$.
						\item The following dependency of the initial data is valid, for $t\leq s\leq \bar T$ \begin{equation}
						\label{Estim2cal}
						\begin{split}
						&\mathbb{E}\pp{\normi{\mathcal{X}(s)-\xi}{-1}{1}^2+\normi{\mathcal{Y}(s)-\eta}{-1}{2}^2}\\[4pt]
						&\leq C\pr{1+\mathbb{E}\pp{\norm{\xi}_{\Lo{2}{1}}^2+\norm{\eta}_{\Lo{2}{1}}^2+\int_t^s\pr{\normi{\phi_1(l)}{-1}{1}^2+\normi{\phi_2(l)}{-1}{2}^2}dl}}(s-t).
						\end{split}
						\end{equation}
					\end{enumerate}
				\end{proposition}
				A quick look at the proof of Proposition \ref{prop1} is enough to convince the reader of the validity of this result. In fact, since we do not wish for estimates sharper than $s-t$,  we do not need to use the pairing $\Ho{1},\Hmo{1}$ and can stay with $\Lo{2}{1}$, hence the constant depending only on $\norm{\xi}_{\Lo{2}{1}}^2$ instead of $\norm{\beta_1(\xi)}_{\Ho{1}}^2$.  
				We recall that \[\norm{\Gamma_1(s,t)}_{\mathcal{L}\pr{\Lo{2}{1};\Lo{2}{1}}}\leq \norm{\Gamma_1(s,t)}_{\mathcal{L}\pr{\Hmo{1};\Lo{2}{1}}}.\]A glance at the proof of Proposition \ref{prop1} shows that $
				\mathbb{E}\pp{\int_0^T\norm{\Gamma_1(t,s)\pr{\mathcal{X}(s)-\xi}}_{\Lo{2}{1}}^2ds}$ is finite such that $\Gamma_1(t,s)\mathcal{X}(s)\in\Lo{2}{1}, \ \mathbb{P}\times ds$-a.s.  on $\Omega\times\pp{0,T}$. Since $e^{sB_1}$ keeps $\Lo{2}{1}$ invariant, we deduce that $\tilde{X}(s)\in\Lo{2}{1}, \ \mathbb{P}\times ds$-a.s. \\
				For full details, the reader is referred to the actual stochastic version (in a more complicated setting) in \cite[Proposition 16 and A2]{viability}.
				
				\begin{theorem}\label{ThAppSol}
					We assume that $K\subset\Hmo{1}\times\Hmo{2}$ is a closed set enjoying the quasi-tangency condition\footnote{in the $\mathbb{L}^2$ setting} on some interval $\pp{0,T}$. Then, for every initial time $t\in\left[0,T\right)$, for every initial data $\pr{\xi,\eta}\in\mathbb{K}_t$ (and $\Lo{2}{1}\times\Lo{2}{2}$-valued), every time horizon $\tilde{T}\in\pp{t,T}$ and every $\varepsilon\in\pr{0,1}$, there exists a constrained $\varepsilon$-approximate solution with $\bar{T}=T$.
				\end{theorem}
				\begin{proof}\\
					\medskip
					\textbf{Step 1.}\\
					One fixes $\varepsilon>0$ and pick $0<\varepsilon'<\varepsilon$ (to be made precise at the end of this step).
					The quasi-tangency condition yields the existence of some admissible $u$ (hence the condition 3.  in Definition \ref{DefEpsSol}, $\delta<\varepsilon'$ and
					$
					p^i\in \mathbb{L}^2\pr{\Omega,\mathcal{F}_{t+\delta},\mathbb{P};\pr{\Lo{2}{i}}}$
					such that \begin{align}\label{est1_0}\mathbb{E}\pp{\normi{p^1}{-1}{1}^2+\normi{p^2}{-1}{2}^2}\leq \delta^2\varepsilon',\textnormal{ and } \pr{\XX\pr{t+\delta}+p^1,\YY\pr{t+\delta}+p^2}\in \mathbb{K}_{t+\delta}.
					\end{align}
					One sets $\phi_i(s)=\frac{1}{\delta}p^i$ to get 
					$\mathbb{E}\pp{\int_t^{t+\delta}\pr{\normi{\phi_1(l)}{-1}{1}^2+\normi{\phi_2(l)}{-1}{2}^2}dl}\leq \delta\varepsilon'$, then $\tau(s)=t, \ \forall s\in\pp{t,t+\delta}$ such that $\pr{\mathcal{X}(s),\mathcal{Y}(s)}=\pr{\xi,\eta}\in K$, $\mathbb{P}$-a.s.  The second assertion in \eqref{est1_0} and the construction of $\phi$ and $\psi$ guarantee that $\pr{\mathcal{X}(t+\delta),\mathcal{Y}(t+\delta)}\in\mathbb{K}_{t+\delta}$.  Finally, owing to \eqref{Estim2cal}, one has
					\begin{align*}
					&\mathbb{E}\pp{\normi{\mathcal{X}(s)-\mathcal{X}(\tau(s))}{-1}{1}^2+\normi{\mathcal{Y}(s)-\mathcal{Y}(\tau(s))}{-1}{2}^2}\\[4pt]&\leq C\pr{1+\mathbb{E}\pp{\norm{\xi}_{\Lo{2}{1}}^2+\norm{\eta}_{\Lo{2}{1}}^2+\int_t^{t+\delta}\pr{\normi{\phi_1(l)}{-1}{1}^2+\normi{\phi_2(l)}{-1}{2}^2}dl}}\delta\\[4pt]
					&\leq C\pr{1+\mathbb{E}\pp{\norm{\xi}_{\Lo{2}{1}}^2+\norm{\eta}_{\Lo{2}{1}}^2}+\varepsilon'}\delta\leq C\pr{1+\mathbb{E}\pp{\norm{\xi}_{\Lo{2}{1}}^2+\norm{\eta}_{\Lo{2}{1}}^2}}\varepsilon'\leq \varepsilon,
					\end{align*}
					explaining the judicious choice of $\varepsilon'$.\\

					\textbf{Step 2.}\\ 
					The family of $\varepsilon$-approximate solutions is non-empty, denoted by $\mathcal{A}$ and endowed with a partial order relation 
					$\pr{\bar T^1,\tau^1,u^1,\phi_1^1,\phi_2^1,\pr{\mathcal{X}^1,\mathcal{Y}^1}}\precsim\pr{\bar T^2,\tau^2,u^2,\phi_1^2,\phi_2^2,\pr{\mathcal{X}^2,\mathcal{Y}^2}}$, if $t\leq\bar T^1\leq\bar T^2$, $u^1=u^2, \ \tau^1=\tau^2,\ \phi_i^1=\phi_i^2,\ \psi_i^1=\psi_i^2$ on $\pp{t, \bar T^1}\times\Omega$. It is clear that every increasing sequence (indexed by a superscript $n\geq 1$) in $\mathcal{A}$ has a maximum in $\mathcal{A}$. \\
					To see this, one naturally defines $\bar T:=\underset{n\geq 1}{\sup}\bar T^n$, and extends $$
					\tau(s):=\begin{cases}\tau^n(s),&\textnormal{ if }s\in\pp{t,\bar T^n};\\
					\sup_{n\geq 1}\tau^n\pr{\bar T^n},&\textnormal{ if }s=\bar T,\end{cases} 
					$$ keeping the properties in 2. in Definition \ref{DefEpsSol}.\\
					Similarly, one extends $u^n$ by picking some $u_0\in U$ and setting $u(\bar{T})=u_0$. The error terms $\phi_1^1,\phi_2^1$ are extended by setting them to $0$ at $s=\bar T$.  The inequality in item 4. in Definition \ref{DefEpsSol} are guaranteed by Fatou's lemma. The fact that $\pr{\mathcal{X},\mathcal{Y}}$ extends $\pr{\mathcal{X}^n,\mathcal{Y}^n}$ is a mere consequence of the uniqueness in \eqref{EqXYCal}.  The continuity of $\mathcal{X},\mathcal{Y}$ (see last item in Proposition \ref{PropPropCal}) and the convergence $\underset{n\rightarrow\infty}{\lim}\tau\pr{\bar T^n}=\tau(\bar T)$ and $\underset{n\rightarrow\infty}{\lim}\bar T^n=\bar T$, together with the closedness of $K$ show that $K\ni\pr{\tilde{X}\pr{\bar T^n},\tilde{Y}\pr{\bar T^n}}$ and $K\ni\pr{\tilde{X}\pr{\tau\pr{\bar T^n}},\tilde{Y}\pr{\tau\pr{\bar T^n}}}$. The same continuity in $\mathbb{L}^2\pr{\Omega, \mathcal{F},\mathbb{P};\Hmo{1}\times\Hmo{2}}$ allows one to pass to the limit as $n\rightarrow\infty$ in the upper-estimate 
					\begin{equation*}
					\mathbb{E}\pp{\normi{\mathcal{X}(\tau^n(s))-\mathcal{X}(s)}{-1}{1}^2+\normi{\mathcal{Y}(\tau^n(s))-\mathcal{Y}(s)}{-1}{2}^2}\leq \varepsilon,
					\end{equation*} to complete the proof. \\
					
					\textbf{Step 3.}\\
					We introduce the function $\n{\cdot}:\mathcal{A}\rightarrow\mathbb{R}_+$ defined by $\n{\pr{\bar T,\tau,u,\pr{\phi_1,\phi_2},\pr{\psi_1,\psi_2},\pr{\mathcal{X},\mathcal{Y}}}}:=\bar T.$ One then concludes owing to Brézis-Browder Theorem (e.g.  \cite[Theorem 2.1.1]{16}) to deduce the existence of a $\n{\cdot}$-maximal element of $\mathcal{A}$. Should this maximal element be given for $\bar{T}^*<\tilde{T}$,  one applies again the first step, thus arriving to a contradiction.
			\end{proof} }

			\section*{Acknowledgement} {{I.C. was partially supported by the Normandie Regional Council (via the M2SiNum project) and by the French ANR grant ANR-18-CE46-0013 QUTE-HPC. }}D.G.  acknowledges support from the National Key R and D Program of China (No. 2018YFA0703900) and the NSF of P.R. China (No. 12031009). I.M. would like to thank the colleagues at the LMI, Normandie University, INSA de Rouen Normandie, for a pleasant stay at their department where part of this work was done.

		\end{document}